\documentclass[11pt]{article}

\usepackage[top=1in, left=1in, bottom=1.5in, right=1.5in]{geometry}

\usepackage{hyperref}  
\usepackage{setspace} 

\usepackage{graphicx}
\usepackage{amsfonts}
\usepackage{amssymb}
\usepackage{amsmath}
\usepackage{amsthm}

\newtheorem{theorem}{Theorem}[section]

\newtheorem{lemma}[theorem]{Lemma}
\newtheorem{conjecture}[theorem]{Conjecture}

\newcommand{\del}{\backslash}

\newcommand{\cB}{\mathcal{B}}

\title{On tree decompositions whose trees are subgraphs}

\author{Rong Chen\ \ \ \ Enzi Liao\\
\\
Center for Discrete Mathematics,\ \ Fuzhou University\\
Fuzhou,\ \ P. R. China}

\begin{document}

\maketitle

\footnotetext[1]{

Email: rongchen@fzu.edu.cn
}

\begin{abstract}
Fix $k \in \mathbb{N}$ and let $G$ be a connected graph with treewidth at most $k$.
We say that $xy \notin E(G)$ is a {\em $k$-ghost-edge} of $G$ if for every tree decomposition $(T, \cB)$ of $G$ with width at most $k$, both $x$ and $y$ are contained in a bag of $(T, \cB)$. Moreover, if $G$ does not contain any $k$-ghost-edges, then  $G$ is {\em $k$-ghost-free}.
Hickingbotham proposed a conjecture that every connected $k$-ghost-free graph $G$ has a tree decomposition $(T, \cB)$ with width at most $k$ such that $T$ is a subgraph of $G$.
In this paper, we prove that Hickingbotham's conjecture is false for all $k\geq3$.

\end{abstract}

\section{Introduction}

All graphs considered in this paper are finite and simple.
Tree decompositions and path decompositions are fundamental objects in graph theory. If a graph $G$ has large pathwidth and there is a tree decomposition $(T,\cB)$ of $G$ with small width and such that $T$ is a subtree of $G$, then $T$ will also have large pathwidth \cite{Hick19}. The crucial question is when does a graph $G$ have a tree decomposition with small width that is indexed by a subtree of $G$. More generally speaking, suppose that a graph $G$ has small treewidth, and consider all tree decompositions $(T,\cB)$ whose width is not too much larger than the optimum. To what extent can we choose or manipulate the ``shape'' of $T$?

For graphs with no long path, we can choose $T$ to also have no long path; this gives rise to
the parameter called treedepth \cite{NdP12}. Similarly, for graphs of bounded degree, we can choose $T$ to also have bounded degree \cite{DO95}; this relates to the parameters of congestion and dilation. Moreover, for graphs excluding any tree as a minor, we can choose $T$ to just be a path; this results in the parameter called pathwidth \cite{BRST91}. 
These results suggest a natural question: can we always find a tree decomposition whose tree $T$ is closely related to $G$ itself, while keeping the width within a function of the treewidth?
In 2019, Dvo\v{r}\'{a}k suggested one way of accomplishing this goal.

\begin{conjecture}(\cite{D19})\label{conj:1.1}
There exists a polynomial function $f$ such that every connected graph $G$ has a tree decomposition $(T, \cB)$ of width at most $f(tw(G))$ such that $T$ is a subgraph of $G$.
\end{conjecture}

When $G$ is a tree, Conjecture~\ref{conj:1.1} is obviously true.
Hickingbotham is the first one to study Conjecture~\ref{conj:1.1}.
Although he found a family of graphs satisfying Conjecture~\ref{conj:1.1}, he conjectured Conjecture~\ref{conj:1.1} wrong for graphs with treewidth 2. Blanco et al. in \cite{BCHHIM23} independently prove that Conjecture~\ref{conj:1.1} is not true for graphs with treewidth 2.


Ghost edges is a tool defined by Hickingbotham in \cite{Hick19} to construct graphs satisfying Conjecture~\ref{conj:1.1}.
Fix $k \in \mathbb{N}$ and let $G$ be a connected graph with $tw(G) \leq k$.
We say that $xy \notin E(G)$ is a {\em $k$-ghost-edge} of $G$ if for every tree decomposition $(T, \cB)$ of $G$ with width at most $k$, we have $x, y \in B_t$ for some $t \in V(T)$.
Although $k$-ghost-edges of $G$ are not in $E(G)$, they behave like real edges with respect to tree decomposition with width at most $k$. For any graph $G$ with treewidth at most $k$ and $xy\notin E(G)$, when there are at least $k+1$ internally vertex disjoint $(x,y)$-paths, Hickingbotham in \cite{Hick19} proved that $xy$ is a $k$-ghost-edge of $G$; while when there are at most $k$ internally vertex disjoint $(x,y)$-paths, he conjectured that $xy$ is not a $k$-ghost-edge of $G$, which was disproved in \cite{Chen} by the first author. 
If a graph $G$ with $tw(G) \leq k$ does not contain any $k$-ghost-edges, we say that $G$ is {\em $k$-ghost-free}. Although Hickingbotham conjectured Conjecture~\ref{conj:1.1} wrong, but he thought Conjecture~\ref{conj:1.1} is true for all connected $k$-ghost-free graphs.

\begin{conjecture}(\cite{Hick19}, Conjecture7.4.2.)\label{conj:1.3}
Fix $k \in \mathbb{N}$.
For every connected $k$-ghost-free graph $G$, there exists a tree decomposition $(T, \cB)$ with width at most $k$ such that $T$ is a subgraph of $G$.
\end{conjecture}

In this paper, we construct infinitely many counterexamples to Conjecture \ref{conj:1.3}, and prove 

\begin{theorem}
\label{thm:1.4}
Let $k,r$ be integers with $k \geq 3$ and $r\geq k+3$. There exists a connected $k$-ghost-free graph $G_{k,r}$ with treewidth $k$ such that if $(T, \cB)$ is a tree decomposition of $G_{k,r}$ such that $T$ is a minor of $G_{k,r}$, then $(T, \cB)$ has width at least $r-2$.
\end{theorem}


\section{Preliminaries}
For a positive integer $n$, set $[n]:= \{1, 2, \dots, n\}$.
The vertex set and edge set of a graph $G$ are denoted by $V(G)$ and $E(G)$, respectively.
For a vertex $u \in V(G)$, let $N_G(u)$ denote the set of neighbors of $u$ in $G$.
The complement of $G$ is denoted by $G^c$.

Let $G$ be a graph and $T$ be a tree.
Let $\cB:=\{B_t \subseteq V(G): t \in V(T)\}$ be a set of subsets of $V(G)$ indexed by the vertices of $T$.
Each $B_x$ is called a \emph{bag}.
Set $$T_v:=T[\{t \in V(T): v \in B_t\}].$$
The pair $(T, \cB)$ is a \emph{tree-decomposition} of $G$ if
\begin{enumerate}
    \item[(T1)] for each edge $ uv \in E(G)$, there is a vertex $t \in V(T)$ such that $u,v \in B_t$; and
    \item[(T2)] for each vertex $v\in V(G)$, the subgraph $T_v$ of $T$ is non-empty and connected.
\end{enumerate}
The \emph{width} of $(T, \cB)$ is then the maximum, over all $t\in V(T)$, of $|B_t|-1$.
The \emph{treewidth} of $G$, denoted by $tw(G)$, is the minimum width of a tree decomposition of $G$.

\begin{figure}[htbp]
\begin{center}
\includegraphics[height=7cm, page=1]{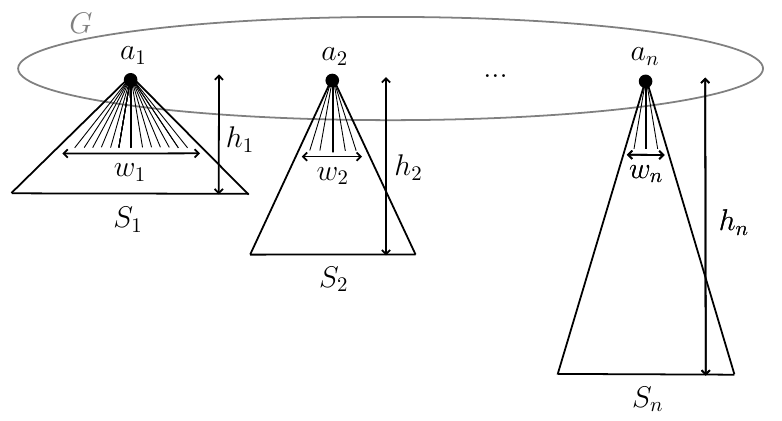}
\caption{The tower-tree graph $\widetilde{G}$ of $G$ obtained by attaching the complete $w_j$-ary tree $S_j$ of height $h_j$ to each vertex $a_j \in V(G)$.}
\label{tree}
\end{center}
\end{figure}

Fix a positive integer $k$, a graph $G$, and an arbitrary ordering $a_1, \ldots, a_n$ of the vertices of $G$. Let $\widetilde{G}$ be the graph which is constructed from $G$ as follows.
\begin{itemize}
\item First define integers $2 = h_1 \ll h_2 \ll \dots \ll h_n$ as follows.
Given $h_{j - 1}$, we define $h_j := (k + 2)^{2h_{j - 1}} + 1$.
\item Next define integers $(k + 1)n + 1 = w_n \ll w_{n - 1} \ll \dots \ll w_1$ and corresponding rooted trees $S_n, S_{n - 1}, \ldots, S_1$ as follows.
Given $w_j$, define $S_j$ to be the complete rooted $w_j$-ary tree of height $h_j$.
Then, given $w_n, w_{n - 1}, \ldots, w_{j + 1}$ and $S_n, S_{n - 1}, \ldots, S_{j + 1}$, define
\[w_j := (k + 1) \left( n + \sum_{i=j+1}^n |V(S_i)| \right) + 1.\]
\end{itemize}
Finally, let $\widetilde{G}$ be the graph which is obtained from the disjoint union of $G, S_1, S_2, \ldots, S_n$ by, for each $j \in [n]$, identifying $a_j$ with the root of $S_j$.
The resulting graph $\widetilde{G}$ is called the \emph{tower-tree} graph of $G$ with respect to $k$ and the ordering $a_1, \ldots, a_n$ of $V(G)$.
Note that the tower-tree graph $\widetilde{G}$ can be obtained from $G$ by adding pendant vertices one at a time.
Evidently, $tw(\widetilde{G}) = \max \{tw(G),1\}$. 

Lemmas \ref{lem:2.1} and \ref{lem:2.1} will be only used in the proof of Theorem \ref{thm:final}.

\begin{lemma}\label{lem:2.1}(\cite{BCHHIM23}, Lemma 3.3.)
Let $k$ be a positive integer, and let $G$ be a connected graph.
Let $\widetilde{G}$ be the tower-tree graph of $G$ with respect to $k$ and an ordering $a_1, \ldots, a_n$ of $V(G)$.
Suppose that $\widetilde{G}$ has a tree decomposition $(T', \cB')$ of width at most $k$ such that $T'$ is a spanning tree of $\widetilde{G}$.
Then there exists a tree decomposition $(T, \cB)$ of $G$ of width at most $k+1$ such that $T$ is a spanning tree of $G$ and for every $v \in V(G)$, we have $v \in T_v$.
\end{lemma}

\begin{lemma}
\label{lem:2.2}
Let $k$ be a positive integer, and let $G$ be a connected graph.
Let $\widetilde{G}$ be the tower-tree graph of $G$ with respect to $k$ and an ordering $a_1, \ldots, a_n$ of $V(G)$.
If $G$ is k-ghost-free, then so is $\widetilde{G}$. 
\end{lemma}
\begin{proof}
Let $S_1, S_2, \ldots, S_n$ be defined as the definition of $\widetilde{G}$, except that we view each tree $S_j$ as an induced subgraph of $\widetilde{G}$, which is rooted as $a_j$.
For $j\in [n]$, let $$\cB_j:=\{\text{the union of}\ u\ \text{and its parent in}\ S_j:\ u \in V(S_j)\}.$$
Note that the parent of $a_j$ in $S_j$ is empty.
Then $(S_j, \cB_j)$ is a tree decomposition of $S_j$.

Now we prove that $\widetilde{G}$ is $k$-ghost-free.
Since $G$ is $k$-ghost-free, for any $xy \notin E(G)$, there exists a tree decomposition $(T, \cB)$ of $G$ with width $\leq k$ such that $x, y$ are not in a bag of $(T, \cB)$.
For any $i \in [n]$, assume that $t_i$ is a vertex of $T$ with $a_i \in B_{t_i}$.
Let $T'$ be the graph which is obtained from a disjoint union of $T, S_1, S_2, \ldots, S_n$ by adding edges $t_ia_i$ for each $i \in [n]$. Set $\cB':= \cB \cup \bigcup_{j=1}^{n}\cB_j$.
Clearly, $(T', \cB')$ is a tree decomposition of $\widetilde{G}$ with width $\leq k$ and satisfies that
\begin{itemize}
\item $x, y$ are not in a bag of $\cB'$, so $xy$ is not a $k$-ghost-edge of $\widetilde{G}$; and
\item for every $i \in [n]$ and every $uv \in E(S_i^c)$, $u$ and $v$ are not in a bag of $\cB'$, so $uv$ is not a $k$-ghost-edge of $\widetilde{G}$; and
\item for every $uv \in E(\widetilde{G}^{c}) \setminus (E(G^c) \cup \bigcup_{j=1}^{n} E(S_{j}^{c}))$, the vertices $u$ and $v$ are not in a bag  of $\cB'$, so $uv$ is not a $k$-ghost-edge of $\widetilde{G}$.
\end{itemize}
Hence, $\widetilde{G}$ is $k$-ghost-free.
\end{proof}

\section{Proof of Theorem~\ref{thm:1.4}}

For each integer $1\leq i\leq2$, let $(G_i, u_i, v_i)$ denote a graph $G_i$ with a special pair of vertices $u_i, v_i\in V(G_i)$ and $V(G_1) \cap V(G_2)=\emptyset$. We say that $(G, u, v) $ is a \emph{parallel connection} of $(G_1, u_1, v_1)$ and $(G_2, u_2, v_2)$, denoted by $(G_1, u_1, v_1; G_2, u_2, v_2)$, if $G$ is obtained from $G_1 \cup G_2$ by adding two new vertices $u, v$ such that $N_G(u)=\{u_1, u_2\}$ and $N_G(v)=\{v_1, v_2\}$, see Figure \ref{H}.


\begin{figure}[htbp]
\begin{center}
\includegraphics[height=3cm, page=2]{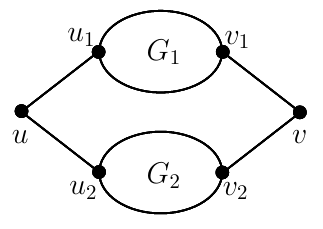}
\caption{A parallel connection of $(G_1, u_1, v_1)$ and $(G_2, u_2, v_2)$.}
\label{H}
\end{center}
\end{figure}

We say that $(G, u, v)$ is \emph{strongly $k$-ghost-free} if the following statements hold.
\begin{enumerate}
\item[(1)] $G$ is $k$-ghost-free.
\item[(2)] For each $xy \notin E(G) \cup \{uv\}$, there exists a tree decomposition $(T,\cB)$ of $G$ with width $\leq k$ such that $x, y$ are not in a bag of $(T,\cB)$, and such that $u, v$ are in a bag.
\item[(3)] When $uv\notin E(G)$, there exists a tree decompositions $(T, \cB)$ of $G$ with width $\leq k$, with $u \in B_{a}$, $v \in B_{b}$, and with $ab \in E(T)$, and with $|B_{a}|, |B_{b}| \leq k$,  and such that $u, v$ are not in a bag.
\end{enumerate}
Note that for any two distinct vertices $u, v$ of a complete graph $K_{k+1}$ on $k+1$ vertices, both $(K_{k+1},u,v)$ and $(K_{k+1}\del uv,u,v)$ are strongly $k$-ghost-free.

\begin{lemma}\label{lem:3.1}
Let $k \geq 3$ be an integer.
If $(G_i, u_i, v_i)$ is strongly $k$-ghost-free for each integer $1\leq i\leq2$, then so is their parallel connection $(G, u, v)$. 
\end{lemma}
\begin{proof}
For each integer $1\leq i\leq2$ and each $x_iy_i \notin E(G_i) \cup \{u_iv_i\}$, let $(T_i, \cB_i)$ be a tree decomposition of $G_i$ with width $\leq k$  such that $x_i, y_i$ are not in a bag of $(T,\cB)$, and such that $u_i, v_i \in B_{s_i}$ for some $s_i \in V(T_i)$; and
let $(T_{2+i}, \cB_{2+i})$ be a tree decompositions of $G_i$ with width $\leq k$, with $u_i \in B_{a_i}$, $v_i \in B_{b_i}$, and with $a_ib_i \in E(T_{2+i})$, and with $|B_{a_i}|, |B_{b_i}| \leq k$,  and such that $u_i, v_i$ are not in a bag. Let $P:=abcd$ be a 4-vertex path.
Set
\[B_{a}:= \{u, u_1, v_1\},\ B_b:=\{u, v, v_1\},\ B_c:= \{u, v, v_2\},\ B_d:= \{u, u_2, v_2\},\]
\[\  B'_b:= \{u, v_1, u_2\},\ B'_c:= \{v, v_1, u_2\},\ B'_d:= \{v, u_2, v_2\},\]
\[\cB:=\{B_a, B_b, B_c, B_d\}\cup \cB_1\cup\cB_2,\ \ \cB':=\{B_a, B'_b, B'_c, B'_d\}\cup \cB_1\cup\cB_2.\]
Let $T$ be the tree obtained from $P \cup T_1\cup T_2$ by adding edges $s_1a$ and $s_2d$.
Then $(T, \cB)$ and $(T, \cB')$ are tree decompositions of $G$ with width $\leq k$ and satisfy that
\begin{itemize}
\item (3) holds for $G$ by considering $(T, \cB')$; and
\item for each integer $1\leq i\leq2$, no bag of $(T, \cB)$ contains $x_i,y_i$; and
\item $u, v$ are in a bag of $(T, \cB)$; and
\item no bag of $(T, \cB)$ contains $x,y$ for any $x \in V(G_1)$ and $y \in V(G_2)$; and 
\item no bag of $(T, \cB)$ contains $v, x$ for any $x \notin N_{G}(v) \cup \{u\}$.  
\end{itemize}
Moreover, by the symmetry between $u$ and $v$, there is a tree decomposition of $G$ with width $\leq k$ such that $u, x$ are not in a bag for any $x \notin N_{G}(u)$.
Hence, to prove that (1) and (2) are true, by the symmetry between $u_1, v_1$ and $u_2, v_2$, it suffices to show that 
\begin{itemize}
\item[(4)] when $u_1v_1 \notin E(G)$, there exists a tree decomposition $(T, \cB)$ of $G$ with width $\leq k$ such that $u_1,v_1$ are not in a bag of $(T, \cB)$, and $u, v$ are in a bag.
\end{itemize}

Let $P=ab$ be a 2-vertex path. Set
\[B_a: = \{u, v, u_1, u_2\},\ \ B_b: = \{v, u_2, v_2\}.\]
Let $T$ be a tree obtained from $P \cup T_2\cup T_3$ by adding edges $a_1a$ and $s_2b$.
Let $\cB'_3$ be obtained from $\cB_3$ by replacing $B_{a_1}$ with $B_{a_1}\cup\{v\}$ and replacing $B_{b_1}$ with $B_{b_1}\cup\{v\}$. Then $(T, \cB'_3 \cup \{B_a, B_b\})$ is a tree decomposition of $G$ satisfying (4) as $|B_{a_1}|, |B_{b_1}| \leq k$.
\end{proof}

Given $(G_0, u_0, v_0)$, define \[(G_r, u_r, v_r): = (G_{r-1}, u_{r-1}, v_{r-1}; G_{r-1}, u_{r-1}, v_{r-1}),\] for any integer $r\geq1$. We say that $(G_r, u_r, v_r)$  is a \emph{$r$-th parallel connection} of $(G_0, u_0, v_0)$. When there is no need to emphasis $(G_0, u_0, v_0)$ (or $u_0, v_0$), we will only say that $(G_r, u_r, v_r)$  is a $r$-th parallel connection of some graph (or $G_0$). 

Let $T$ be a spanning tree of a graph $G$. For any $e \in E(G) \setminus E(T)$, let $C_T^e$ denote the unique cycle in $T\cup\{e\}$ containing $e$. 

The proofs of Lemmas \ref{lem:3.2} and \ref{lem:3.3} follow ideas from (\cite{BCHHIM23}, Lemma 4.4.).

\begin{lemma}\label{lem:3.2}
Let $(G_r, u_r, v_r)$  be a $r$-th parallel connection of some graph $G$ and $T$ a spanning tree of $G_r$. Then there exists a matching $M \subseteq E(G_r) \setminus E(T)$ of size $r$ such that
\[\bigcap_{e \in M} V(C_T^e) \neq \emptyset.\]
\end{lemma}
\begin{proof}
Let $P$ be the unique $(u_r, v_r)$-path in $T$. To prove the lemma, it suffices to show that {\bf (a)} there exists a matching $M \subseteq E(G_r) \setminus E(T)$ of size $r$ such that $\bigcap_{e \in M}E(C_T^e) \cap E(P) \neq \emptyset$.

We prove (a) by induction on $r$.
Let $C_1, C_2$ be the components of $G_r\del\{u_r, v_r\}$. Then $C_1, C_2$ are isomorphic to $G_{r-1}$, which is the $(r-1)$-th parallel graph of  $G$.
Without loss of generality assume that all interior vertices of $P$ are contained in $C_1$. Then $T\del V(C_1)$ has exactly two components $T_1,T_2$ such that $|V(T_1)\cap\{u_r, v_r\}|=|V(T_2)\cap\{u_r, v_r\}|=1$. Let $e_r$ be an edge of $G_r$ linking  $T_1,T_2$. Then $V(e_r)\cap V(C_1)=\emptyset$ and $P\subset C_T^{e_r}$. Hence, when $r=1$, set $M:=\{e_r\}$, implying (a) holding. Assume that $r>1$. Set $T':=T\del (V(C_2)\cup\{u_r, v_r\})$. 
Since $T'$ is a spanning tree of $C_1$ that is isomorphic to $G_{r-1}$, by induction, there exists a matching $M' \subseteq E(C_1) \setminus E(T')$ of size $r-1$ such that $\bigcap_{e \in M'}E(C_{T'}^e) \cap E(P\del \{u_r, v_r\}) \neq \emptyset$. Hence, $\bigcap_{e \in M'\cup\{e_r\}}E(C_{T'}^e) \cap E(P) \neq \emptyset$ as $P\subset C_T^{e_r}$. Moreover, since $V(e_r)\cap V(C_1)=\emptyset$, we have that $M'\cup\{e_r\}$ is a matching of $G_r$ of size $r$, so (a) holds.
\end{proof}

\begin{lemma}\label{lem:3.3}
Let $G_r$  be a $r$-th parallel graph of some graph.
If $(T, \cB)$ is a tree decompsition of $G_r$ such that $T$ is a spanning tree of $G_r$, and $v \in T_v$ for every $v \in V(G_r)$, then the width of $(T, \cB)$ is at least $r-1$.
\end{lemma}
\begin{proof}
Let $M \subseteq E(G_r) \setminus E(T)$ be a matching of size $r$ satisfying Lemma~\ref{lem:3.2}. Set $M := \{a_1b_1, \dots, a_rb_r\}$ and let $x \in \bigcap_{i=1}^r V(C_T^{a_ib_i})$. Since $a_i \in V(T_{a_i})$ and $b_i \in V(T_{b_i})$ for every $i \in [r]$, by (T1) and (T2), the unique $(a_i, b_i)$-path in $T$ is contained in $T_{a_i} \cup T_{b_i}$, that is, $V(C_T^{a_ib_i})\subseteq V(T_{a_i} \cup T_{b_i})$; hence, $a_i \in B_x$ or $b_i \in B_x$. So $|B_x| \geq r$, implying that the width of $(T, \cB)$ is at least $r-1$.
\end{proof}

\begin{lemma}\label{lem:2.5}(\cite{BCHHIM23}, Lemma 2.5.)
If $G$ is a connected graph with a tree decomposition $(T, \cB)$ with width $k$ such that $T$ is a minor of $G$, then there is a tree decomposition $(T', \cB')$ of $G$ with width $k$ such that $T'$ is a spanning tree of $G$.
\end{lemma}

We are now ready to prove Theorem \ref{thm:1.4}, which is restated here for convenience.

\begin{theorem}\label{thm:final}
Let $k,r$ be integers with $k \geq 3$ and $r\geq k+3$. Then there exists a connected $k$-ghost-free graph $G_{k,r}$ with treewidth $k$ such that if $(T, \cB)$ is a tree decomposition of $G_{k,r}$ such that $T$ is a minor of $G_{k,r}$, then $(T, \cB)$ has width at least $r-2$.
\end{theorem}
\begin{proof}
Let $G_0$ be a graph with $u_0, v_0\in V(G_0)$ such that $(G_0,u_0, v_0)$ is strongly $k$-ghost-free. 
Let $(G_r,u_r, v_r)$ be a $r$-th parallel graph of $(G_0,u_0, v_0)$. Since $(G_0,u_0, v_0)$ is strongly $k$-ghost-free, so is $(G_r,u_r, v_r)$ by  Lemma~\ref{lem:3.1}. Moreover, since $tw(G_0) = k$, we have $tw(G_r) = k$. Let $\widetilde{G}_r$ be the tower-tree graph of $G_r$ with respect to $k$ and an ordering $a_1, \ldots, a_n$ of $V(G_r)$. Then $tw(\widetilde{G}_r) = \max \{tw(G_r), 1\} = k$. Since $G_r$ is $k$-ghost-free, so is $\widetilde{G}_r$ by Lemma~\ref{lem:2.2}. 

We claim that $\widetilde{G}_r$ satisfies Theorem \ref{thm:final}. Assume not. By Lemma~\ref{lem:2.5}, there is a tree decomposition $(T', \cB')$ of $\widetilde{G}_r$ with width at most $r-3$ such that $T'$ is a spanning tree of $\widetilde{G}_r$. Moreover, by Lemma~\ref{lem:2.1}, there is a tree decomposition $(T, \cB)$ of $G_r$ with width $\leq r-2$ such that $T$ is a spanning tree of $G_r$ and $v \in T_v$ for every $v \in V(G_r)$, which is a contradiction to Lemma~\ref{lem:3.3}.
\end{proof}

\section{Acknowledgments}
This research was partially supported by grants from the Natural Sciences Foundation of Fujian Province (grant 2025J01486).


\begin{thebibliography}{}
\bibitem{BCHHIM23}
P. Blanco, L. Cook. M. Hatzel, C. Hilaire, F. Illingworth, R. McCarty, On tree decompositions whose trees are minors, J. Graph Theory, 106 (2023): 296-306.

\bibitem{BRST91}
D. Bienstock, N. Robertson, P. Seymour, and R. Thomas, Quickly excluding a forest, J. Combin. Theory Ser. B. 52 (1991): 274-283.

\bibitem{Chen}
R. Chen, A counterexample to Hickingbotham's conjecture about $k$-ghost-edges, 2026, arXiv:2602.03016v1.

\bibitem{D19}
Z. Dvo\v r\'ak, Problem 20 from the Barbados graph theory workshop in 2019, 2019. https://sites.google.com/
site/sophiespirkl/open-problems/2019-open-problems-for-the-barbados-graph-theory-workshop

\bibitem{DO95}
G. Ding and B. Oporowski, Some results on tree decomposition of graphs, J. Graph Theory. 20 (1995):  481-499.

\bibitem{Hick19}
R. Hickingbotham, Graph minors and tree decompositions, B.Sc. (Honours) thesis, School of Mathematics,
Monash University, Melbourne, Australia, 2019. http://www.roberthickingbotham.com

\bibitem{NdP12}
J. Ne\v set\v ril and de P. O. Mendez, Bounded height trees and treedepth, Springer, Berlin, Heidelberg, 2012, pp. 115-144.






\end{thebibliography}
\end{document}